\documentclass[12pt,a4paper,reqno]{amsart}
\usepackage{amsmath}
\usepackage{amssymb,latexsym}
\usepackage[dvips]{graphicx}
\usepackage{color}

\newcommand{\R}{\mathbb R}

\newcommand{\C}{\mathbb C}

\newtheorem{theorem}{Theorem} [section]

\newtheorem{corollary}{Corollary} [section]
\newtheorem{definition}{Definition} [section]

\newtheorem{remark}{Remark}[section]

\let\ssection=\section\renewcommand{\section}{\setcounter{equation}{0}\ssection}
\begin{document}
\title ["Universal" inequalities for eigenvalues]{''Universal'' inequalities for the eigenvalues of the
biharmonic operator}
\author{ Sa\"{\i}d Ilias and Ola Makhoul}
\date{15 novembre 2009}

\address{S. Ilias, O. Makhoul: Universit\'e Fran\c{c}ois rabelais de Tours, Laboratoire de Math\'ematiques
et Physique Th\'eorique, UMR-CNRS 6083, Parc de Grandmont, 37200
Tours, France} \email{ilias@univ-tours.fr, ola.makhoul@lmpt.univ-tours.fr}


\begin{abstract}
In this paper, we establish universal inequalities for eigenvalues of the clamped plate problem on compact submanifolds of Euclidean spaces, of spheres and of real, complex and quaternionic projective spaces. We also prove similar results for the biharmonic operator on domains of Riemannian manifolds admitting spherical eigenmaps (this includes the compact homogeneous Riemannian spaces) and finally on domains of the hyperbolic space.
\end{abstract}
\maketitle
%%%%%%%%%%%%%%%%%%%%%%%%%%%%%%%%%%%%%%%%%%%%%%%%%%%%%%%%%%%%%%%%%%%%%%%%%
\section{Introduction}
Let $(M,g)$ be a Riemannian manifold of dimension $n$ and let $\Delta$ be the Laplacian operator on $M$.\\
In this paper, we will be concerned with the following eigenvalue problem for Dirichlet biharmonic operator, called the clamped plate problem:
\begin{equation}\label{clamped prob}\begin{cases}
\Delta^2 u=\lambda u & \hbox{in} \,\,\Omega\\
\displaystyle u=\frac{\partial u}{\partial \nu}=0 & \hbox{on}\,\, \partial \Omega,
\end{cases}
\end{equation}
where $\Omega$ is a bounded domain in $M$, $\Delta^2$ the biharmonic operator in $M$ and $\nu$ is the outward unit normal. It is well known that the eigenvalues of this problem form a countable family $0<\lambda_1 \leq \lambda_2 \leq \ldots \rightarrow +\infty$.\\
\indent For the case when $M=\R^n$, in 1956, Payne, Polya and Weinberger \cite{PPW} (henceforth PPW) established the following inequality, for each $k\ge 1$,
\begin{equation*}
\lambda_{k+1}-\lambda_k \leq \frac{8(n+2)}{n^2k}\sum_{i=1}^k\lambda_i.
\end{equation*}
Implicit in the PPW work, as noticed by Ashbaugh in \cite{Ashb}, is the better inequality 
\begin{equation}\label{Ashb}
\lambda_{k+1}-\lambda_k \le \frac{8(n+2)}{n^2k^2}\bigg(\sum_{i=1}^k \lambda_i^{\frac{1}{2}}\bigg)^2.
\end{equation} 
Later, in 1984, Hile and Yeh \cite{HileYeh} extended ideas from earlier work on the Laplacian by Hile and Protter \cite{HileProt} and proved the better bound
\begin{equation*}
 \frac{n^2k^{\frac{3}{2}}}{8(n+2)}\leq \bigg(\sum_{i=1}^k\frac{\lambda_i^{\frac{1}{2}}}{\lambda_{k+1}-\lambda_i}\bigg)\Big(\sum_{i=1}^k\lambda_i\Big)^{\frac{1}{2}}.
\end{equation*}
Implicit in their work is the stronger inequality
\begin{equation*}
 \frac{n^2k^2}{8(n+2)}\leq \bigg(\sum_{i=1}^k\frac{\lambda_i^{\frac{1}{2}}}{\lambda_{k+1}-\lambda_i}\bigg)\Big(\sum_{i=1}^k\lambda_i^{\frac{1}{2}}\Big),
\end{equation*}
which was proved independently by Hook \cite{Hook1} and Chen and Qian \cite{ChenQian3} in 1990 (see also \cite{ChenQian4}, \cite{ChenQian1} and \cite{ChenQian2}).\\In 2003, Cheng and Yang \cite{ChengYang} obtained the following bound
\begin{equation}\label{1}
\sum_{i=1}^k(\lambda_{k+1}-\lambda_i) \leq \bigg[\frac{8(n+2)}{n^2}\bigg]^{\frac{1}{2}}\sum_{i=1}^k\Big[\lambda_i(\lambda_{k+1}-\lambda_i)\Big]^{\frac{1}{2}}.
\end{equation}
Very recently, Cheng, Ichikawa and Mametsuka \cite{Cheng-Ichik-Mamet1} obtained an inequality for eigenvalues of Laplacian with any order $l$ on a bounded domain in $\R^n$. %then in the unit sphere $\mathbb{S}^n(1)$ (\cite{Cheng-Ichik-Mamet2} respectively).
In particular, they showed, for $l=2$,
\begin{equation}\label{2}
 \sum_{i=1}^k (\lambda_{k+1}-\lambda_i)^2\leq \frac{8(n+2)}{n^2}\sum_{i=1}^k(\lambda_{k+1}-\lambda_i)\lambda_i.
\end{equation}
For the case when $M=\mathbb{S}^n$, Wang and Xia \cite{Wang-Xia} showed
\begin{align}\label{WX}
\sum_{i=1}^k(\lambda_{k+1}-\lambda_i)^2
\le\frac{1}{n}\bigg[\sum_{i=1}^k(\lambda_{k+1}&-\lambda_i)^2\Big(n^2+(2n+4)\lambda_i^{\frac{1}{2}}\Big)\bigg]^{\frac{1}{2}}\times\nonumber\\
&\bigg[\sum_{i=1}^k(\lambda_{k+1}-\lambda_i)\
\Big(n^2+4\lambda_i^{\frac{1}{2}}\Big)\bigg]^{\frac{1}{2}},
\end{align}
from which they deduced, using a variant of Chebyshev inequality,
\begin{equation}\label{3}
 \sum_{i=1}^k(\lambda_{k+1}-\lambda_i)^2\leq \frac{1}{n^2}\sum_{i=1}^k(\lambda_{k+1}-\lambda_i)\Big(2(n+2)\lambda_i^{\frac{1}{2}}+n^2\Big)\Big(4\lambda_i^{\frac{1}{2}}+n^2\Big).
\end{equation}
This last inequality was also obtained, by a different method, by Cheng, Ichikawa and Mametsuka (see \cite{Cheng-Ichik-Mamet2}).\\ 
On the other hand, Wang and Xia \cite{Wang-Xia} also considered the problem (\ref{clamped prob}) on domains of an $n$-dimensional complete minimal submanifold $M$ of $\R^m$ and proved in this case
\begin{align}\label{WX2}
 \sum_{i=1}^k(\lambda_{k+1}&-\lambda_i)^2\nonumber\\
\le\bigg(\frac{8(n+2)}{n^2}\bigg)^{\frac{1}{2}}\bigg(\sum_{i=1}^k(\lambda_{k+1}-\lambda_i)^2\lambda_i^{\frac{1}{2}}\bigg)^{\frac{1}{2}}&\bigg(\sum_{i=1}^k(\lambda_{k+1}-\lambda_i)\lambda_i^{\frac{1}{2}}\bigg)^{\frac{1}{2}},
\end{align}
from which they deduced the following generalization of inequality (\ref{2}) to minimal Euclidean submanifolds  
\begin{equation}\label{WX3}
 \sum_{i=1}^k(\lambda_{k+1}-\lambda_i)^{2}\leq \frac{8(n+2)}{n^2}\sum_{i=1}^k(\lambda_{k+1}-\lambda_i)\lambda_i.
\end{equation}
Recently, Cheng, Ichikawa and Mametsuka \cite{Cheng-Ichik-Mamet3} extended this last inequality to any complete Riemannian submanifold $M$ in $\R^m$ and showed 
\begin{equation}\label{CIM}
 \sum_{i=1}^k(\lambda_{k+1}-\lambda_i)^2\le \frac{1}{n^2}\sum_{i=1}^k(\lambda_{k+1}-\lambda_i)\Big(n^2\delta+2(n+2)\lambda_i^{\frac{1}{2}}\Big)\Big(n^2\delta+4\lambda_i^{\frac{1}{2}}\Big),
\end{equation}
with $\delta=\displaystyle{Sup_{\Omega}|H|^{2}}$, where $H$ is the mean curvature of $M$.\\
 The goal of the first section of this article is to study the relation between eigenvalues of the biharmonic operator and the local geometry of Euclidean submanifolds $M$ of arbitrary codimensions. The approach is based on an algebraic formula (see theorem \ref{abstract form 1}) we proved in \cite{Ilias-Mak3}. This approach is useful  for the unification and for the generalization of all the results in the literature. In fact, using this general algebraic inequality, we obtain (see theorem \ref{clamp1}) the following inequality 
 \begin{align}\label{16'}
\sum_{i=1}^k f(\lambda_{i}) \leq \frac{1}{n}\bigg[\sum_{i=1}^k&g(\lambda_{i})\bigg(2(n+2)\lambda_i^{\frac{1}{2}}+n^2\delta\bigg)\bigg]^{\frac{1}{2}}\times \nonumber\\
& \bigg[\sum_{i=1}^k\frac{\big(f(\lambda_i)\big)^2}{g(\lambda_i)(\lambda_{k+1}-\lambda_i)}\bigg(4\lambda_i^{\frac{1}{2}}+n^{2}\delta\bigg)\bigg]^{\frac{1}{2}},
\end{align}
where $f$ and $g$ are two functions satisfying some functional conditions (see Definition \ref{def1}), $\delta=\displaystyle{Sup_{\Omega}|H|^{2}}$ and $H$ is the mean curvature of $M$. We note that the family of such couples of functions is large. And particular choices for $f$ and $g$ lead to the known results. For instance, if we take $f(x)=g(x)=(\lambda_{k+1}-x)^2$, then (\ref{16'}) becomes 
\begin{align}\label{16"}
\sum_{i=1}^k (\lambda_{k+1}-\lambda_{i})^{2} \leq \frac{1}{n}\bigg[\sum_{i=1}^k& (\lambda_{k+1}-\lambda_{i})^{2}\bigg(2(n+2)\lambda_i^{\frac{1}{2}}+n^2\delta\bigg)\bigg]^{\frac{1}{2}}\times \nonumber\\
& \bigg[\sum_{i=1}^k(\lambda_{k+1}-\lambda_i)\bigg(4\lambda_i^{\frac{1}{2}}+n^{2} \delta\bigg)\bigg]^{\frac{1}{2}},
\end{align}
which gives easily (see remark \ref{rem1}) inequality (\ref{CIM}) of Cheng, Ichikawa and Mametsuka.\\
 In the second section, we consider the case of manifolds admitting spherical eigenmaps and obtain similar results. As a consequence, we obtain universal inequalities for the clamped plate problem on domains of any compact homogeneous Riemannian manifold.\\ 
 In the last section, we show how one can easily obtain, from the algebraic techniques used in the previous sections, universal inequalities for eigenvalues of (\ref{clamped prob}) on domains of the hyperbolic space $\mathbb{H}^n$.\\
We also observe that all our results hold if we add a potential to $\Delta^{2}$ (i.e  $\Delta^2+q$ where $q$ is a smooth potential). For instance, in this case instead of inequality (\ref{16'}), we obtain 
\begin{align}\label{16'''}
\sum_{i=1}^k f(\lambda_{i}) \leq \frac{1}{n}\bigg[\sum_{i=1}^k & g(\lambda_{i})\bigg(2(n+2)\overline{\lambda_i}^{\frac{1}{2}}+n^2\delta\bigg)\bigg]^{\frac{1}{2}}\times \nonumber\\
& \bigg[\sum_{i=1}^k\frac{\big(f(\lambda_i)\big)^2}{g(\lambda_i)(\lambda_{k+1}-\lambda_i)}\bigg(4\overline{\lambda_i}^{\frac{1}{2}}+n^{2}\delta\bigg)\bigg]^{\frac{1}{2}},
\end{align}  
where $\overline{\lambda_{i}}= \lambda_{i}- {\rm inf}_{\Omega}\,q$.\\
Finally, note that the case of the clamped problem with weight :
\begin{equation}
\begin{cases}
\Delta^2 u=\lambda\, \rho\, u & \hbox{in} \,\,\Omega\\
\displaystyle u=\frac{\partial u}{\partial \nu}=0 & \hbox{on}\,\, \partial \Omega,
\end{cases}
\end{equation}
can be easily treated with minor changes.
%%%%%%%%%%%%%%%%%%%%%%%%%%%%%%%%%%%%%%%%%%%%%%%%%%%%%%%%%%%%%%%%%%%%%%%
\section{Euclidean Submanifolds}
Before stating the main result of this section, we introduce a family of couples of functions and a theorem obtained earlier in \cite{Ilias-Mak3}, which will play an essential role in the proofs of all our results.
\begin{definition}\label{def1}
Let $\lambda >0$. A couple $(f,g)$ of functions defined on $]0,\lambda[$ belongs to $\Im_{\lambda}$ provided that  
\begin{itemize}
 \item[1.] $f$ and $g$ are positive,
\item[2.] $f$ and $g$ satisfy the following condition,\\ 
 for any $x,\, y \in ]0,\lambda[$ such that $x \neq y$,
\begin{equation}\label{cond}
\Big(\frac{f(x)-f(y)}{x-y}\Big)^2+\Big(\frac{\big(f(x)\big)^2}{g(x)(\lambda-x)}+\frac{\big(f(y)\big)^2}{g(y)(\lambda-y)}\Big)\Big(\frac{g(x)-g(y)}{x-y}\Big)\le 0.
\end{equation}
\end{itemize}
\end{definition}
A direct consequence of our definition is that $g$ must be nonincreasing.\\
If we multiply $f$ and $g$ of $\Im_{\lambda}$ by positive constants the resulting functions are also in $\Im_{\lambda}$. In the case where $f$ and $g$ are differentiable, one can easily deduce from (\ref{cond}) the following necessary condition:
\begin{equation*}
\bigg[\big(\ln{f(x)}\big)'\bigg]^2 \le \frac{-2}{\lambda-x}\big(\ln{g(x)}\big)'.
\end{equation*}
This last condition helps us to find many couples $(f,g)$ satisfying the conditions 1) and 2) above. Among them, we mention
$\left\{\Big(1, (\lambda-x)^{\alpha}\Big)\, / \, \alpha \ge 0\right\},$ $\left\{\Big((\lambda-x),(\lambda-x)^{\beta}\Big)\, / \, \beta\ge \frac{1}{2}\right\},$ $\left\{\Big((\lambda-x)^{\delta},(\lambda-x)^{\delta}\Big)\, / \, 0<\delta \le 2\right\}$.\\ 
\indent Let $\mathcal{H}$ be a complex Hilbert space with scalar product $\langle .,. \rangle$ and
corresponding norm $\|.\|$. For any two operators $A$ and $B$, we denote by $[A,B]$ their commutator, defined by $[A,B]=AB-BA$. 
\begin{theorem}\label{abstract form 1}
Let $A$ : $\mathcal{D}\subset \mathcal{H}\longrightarrow
\mathcal{H}$ be a self-adjoint operator defined on a dense domain
$\mathcal{D}$, which is semibounded below and has a discrete
spectrum $\lambda_1 \leq \lambda_2\leq \lambda_3...$. 
Let $\{T_p:\mathcal{D}\longrightarrow \mathcal{H}\}_{p=1}^n$ be a collection of
skew-symmetric operators, and $\{B_p:T_p(\mathcal{D})\longrightarrow
\mathcal{H}\}_{p=1}^n$ be a collection of symmetric operators,
leaving $\mathcal{D}$ invariant. We denote by
$\left\{u_i\right\}_{i=1}^{\infty}$ a basis of orthonormal
eigenvectors of $A$, $u_{i}$ corresponding to $\lambda_i$. Let $k\ge 1$ and assume that $\lambda_{k+1}>\lambda_k$. 
Then, for any $(f,g)$ in $\Im_{\lambda_{k+1}}$  
\begin{align}\label{abst}
\bigg( \sum_{i=1}^k
\sum_{p=1}^n &f(\lambda_i)\langle[T_p,
B_p]u_i,u_i\rangle \bigg)^2 \\
\le 4 \bigg(\sum_{i=1}^k\sum_{p=1}^n g(\lambda_i)& \langle
[A,B_p]u_i, B_p u_i\rangle \bigg) \bigg(\sum_{i=1}^k
\sum_{p=1}^n \frac{\big(f(\lambda_i)\big)^2}{g(\lambda_i)(\lambda_{k+1}-\lambda_i)}\|T_p u_i\|^2\bigg).\nonumber
\end{align}
\end{theorem}
Our first result is the following application of this inequality to the eiganvalues of the clamped plate problem (\ref {clamped prob}) on a domain of a Euclidean submanifold :
\begin{theorem}\label{clamp1}
Let $X:M\longrightarrow \R^m$ be an isometric immersion of an $n$-dimensional Riemannian manifold $M$ in $\R^{m}$. Let $\Omega$ be a bounded domain of $M$ and consider the clamped plate problem (\ref{clamped prob}) on it. For any $k\ge 1$ such that  $\lambda_{k+1}>\lambda_k$ and for any $(f,g)$ in $\Im_{\lambda_{k+1}}$, we have
\begin{align}\label{14'}
\sum_{i=1}^k f(\lambda_{i}) \leq \frac{2}{n}\bigg[\sum_{i=1}^k&g(\lambda_{i})\bigg(2(n+2)\lambda_i^{\frac{1}{2}}+n^2\delta\bigg)\bigg]^{\frac{1}{2}}\times \nonumber\\
& \bigg[\sum_{i=1}^k\frac{\big(f(\lambda_i)\big)^2}{g(\lambda_i)(\lambda_{k+1}-\lambda_i)}\bigg(\lambda_i^{\frac{1}{2}}+\frac{n^2}{4}\delta\bigg)\bigg]^{\frac{1}{2}},
\end{align}
where $\delta=\displaystyle sup_{\Omega} \arrowvert H\arrowvert^2$ and $H$ is the mean curvature vector field of the immersion $X$ (i.e which is given by $\frac{1}{n} \,{\rm trace}\,h$, where $h$ is the second fundamental form of $X$).
\end{theorem}
\begin{proof}To prove this theorem, we apply inequality (\ref{abst}) of Theorem \ref{abstract form 1} with $A=\Delta^2$, $B_p=X_p$ and $T_p=[\Delta,X_p]$, $p=1,\ldots,m$, where $X_1,\ldots,X_m$ are the components of the immersion $X$.
This gives
\begin{align}\label{15'}
\bigg( \sum_{i=1}^k
\sum_{p=1}^m &f(\lambda_i)\langle[[\Delta,X_p],
X_p]u_i,u_i\rangle_{L^2} \bigg)^2 \\
\le 4 \bigg(\sum_{i=1}^k\sum_{p=1}^m g(\lambda_i)& \langle
[\Delta^2,X_p]u_i, X_p u_i\rangle_{L^2} \bigg) \bigg(\sum_{i=1}^k
\sum_{p=1}^m \frac{\big(f(\lambda_i)\big)^2}{g(\lambda_i)(\lambda_{k+1}-\lambda_i)}\|[\Delta,X_p] u_i\|_{L^2}^2\bigg).\nonumber
\end{align}
where $u_i$ are the $L^2$-normalized eigenfunctions.
First we have,
for any $p=1,\ldots,m$,
\begin{align*}
[\Delta^2, X_p]u_i=& \Delta^2 X_p u_i+2 \nabla \Delta X_p.\nabla u_i+2\Delta (\nabla X_p.\nabla u_i)\nonumber\\
& +2 \Delta X_p \Delta u_i+ 2 \nabla X_p.\nabla \Delta u_i.
\end{align*}
Thus
\begin{align}
\langle [\Delta^2, X_p]u_i, X_p u_i \rangle_{L^2} &= \int_{\Omega}  u_i^2 X_p \Delta^2 X_p + 2\int_{\Omega}X_p u_i \nabla \Delta X_p.\nabla u_i\nonumber\\
&+2\int_{\Omega} X_p u_i \Delta (\nabla X_p.\nabla u_i)+2\int_{\Omega}X_p u_i \Delta X_p \Delta u_i\nonumber\\
&+2\int_{\Omega} X_p u_i \nabla X_p.\nabla \Delta u_i \label{k3'}\\
=&\int_{\Omega} \Delta X_p \Delta \big(X_p u_i^2\big) -2\int_{\Omega}div \big(X_p u_i \nabla u_i\big) \Delta X_p \nonumber\\
&+2\int_{\Omega} \Delta \big(X_p u_i\big) \nabla X_p.\nabla u_i +2\int_{\Omega}X_p \Delta X_p u_i \Delta u_i\nonumber\\
&-2\int_{\Omega}div\big(X_p u_i \nabla X_p \big)\Delta u_i.\nonumber\\
\end{align}
A straightforward calculation gives
\begin{align}\label{l3}
\langle [\Delta^2, X_p]u_i, X_p u_i \rangle_{L^2} =& \,4\int_{\Omega} u_i \Delta X_p \nabla X_p.\nabla u_i + \int_{\Omega} \big(\Delta X_p\big)^2 u_i^2 \nonumber\\
&+ 4\int_{\Omega} \big(\nabla X_p.\nabla u_i\big)^2 -2 \int_{\Omega} \arrowvert \nabla X_p\arrowvert^2 u_i \Delta u_i.
\end{align}
Since $X$ is an isometric immersion, we have
\begin{align}\label{l3'}
n H=(\Delta X_1, \ldots,& \Delta X_m),\,\, \sum_{p=1}^m u_i \Delta X_p \nabla X_p.\nabla u_i =0 \nonumber\\
&\hbox{and}\,\,\sum_{p=1}^m \big(\nabla X_p.\nabla u_i\big)^2=\arrowvert \nabla u_i\arrowvert^2.
\end{align}
Incorporating these identities in (\ref{l3}) and summing on $p$, from 1 to $m$, we obtain
\begin{align}
\sum_{p=1}^m \langle [\Delta^2, X_p]u_i, X_p u_i \rangle_{L^2} = & \,4\int_{\Omega}\arrowvert \nabla u_i\arrowvert^2 -2n\int_{\Omega}u_i \Delta u_i + n^2\int_{\Omega} \arrowvert H\arrowvert^2 u_i^2 \nonumber\\
=&\, 2(n+2) \int_{\Omega} u_i (-\Delta u_i)+ n^2\int_{\Omega} \arrowvert H\arrowvert^2 u_i^2\nonumber\\
\leq  2(n+2)& \bigg[\int_{\Omega} (-\Delta u_i)^2\bigg]^{\frac{1}{2}}\bigg[\int_{\Omega}u_i^2\bigg]^{\frac{1}{2}}+n^2\int_{\Omega} \arrowvert H\arrowvert^2 u_i^2\label{m3}\\
=& 2(n+2)\lambda_i^{\frac{1}{2}}+n^2\int_{\Omega} \arrowvert H\arrowvert^2 u_i^2\nonumber\\
\leq &2(n+2)\lambda_i^{\frac{1}{2}}+n^2\delta \label{n3},
\end{align}
where we used the Cauchy-Schwarz inequality to obtain (\ref{m3}) and where $\delta=sup_{\Omega} |H|^2$.\\
On the other hand, we have
\begin{equation*}
[\Delta, X_p]u_i = 2\nabla X_p.\nabla u_i+ u_i \Delta X_p,
\end{equation*}
then \begin{align*}
\sum_{p=1}^m \|[\Delta,X_p]u_i\|_{L^2}^2 =& \sum_{p=1}^m\int_{\Omega}\Big(2\nabla X_p.\nabla u_i+ u_i \Delta X_p\Big)^2 \nonumber\\
 =& 4\sum_{p=1}^m \int_{\Omega}\Big(\nabla X_p.\nabla u_i\Big)^2 +4\sum_{p=1}^m\int_{\Omega}u_i \Delta X_p \nabla X_p.\nabla u_i \nonumber\\
&+\sum_{p=1}^m\int_{\Omega}(\Delta X_p)^2 u_i^2.
     \end{align*}
Using identities (\ref{l3'}), we obtain
\begin{align}\label{o3}
 \sum_{p=1}^m \|[\Delta,X_p]u_i\|_{L^2}^2 = &4\int_{\Omega}\arrowvert\nabla u_i\arrowvert^2+n^2\int_{\Omega}\arrowvert H\arrowvert^2u_i^2 \nonumber\\
= & 4\int_{\Omega}(-\Delta u_i).u_i +n^2\int_{\Omega}\arrowvert H\arrowvert^2u_i^2\nonumber\\
\leq & 4\bigg[\int_{\Omega} (-\Delta u_i)^2\bigg]^{\frac{1}{2}}\bigg[\int_{\Omega}u_i^2\bigg]^{\frac{1}{2}} +n^2\delta\nonumber\\
=& 4\lambda_i^{\frac{1}{2}}+n^2\delta.
\end{align}
A direct calculation gives
\begin{equation*}
\langle [[\Delta,X_p],X_p]u_i,u_i\rangle_{L^2}= \int_{\Omega} \Big(\Delta(X_p^2u_i) -2X_p \Delta (X_p u_i)+X_p^2\Delta u_i\Big)u_i=2\int_{\Omega}\arrowvert \nabla X_p\arrowvert^2 u_i^2.
\end{equation*}
Therefore
\begin{equation}\label{p3}
\sum_{p=1}^m\langle [[\Delta,X_p],X_p]u_i,u_i\rangle_{L^2}=2\sum_{p=1}^m\int_{\Omega}\arrowvert \nabla X_p\arrowvert^2 u_i^2 =2 n.
\end{equation}
To conclude, we simply use the estimates (\ref{n3}), (\ref{o3}) and (\ref{p3}) together with inequality (\ref{15'}). 
\end{proof}
\begin{remark} \label {rem1}
\begin{itemize}
\item As indicated in the end of the introduction, Theorem \ref{clamp1} holds for a general operator $\Delta^2+q$, where $q$ is a smooth potential. Indeed, this is an immediate consequence of the fact that $[\Delta^2+q,X_p]=[\Delta^2,X_p]$ and all the proof of Theorem \ref{clamp1} works in this situation. The only modification is in the estimation of the term $\int_{\Omega}\arrowvert \nabla u_i\arrowvert^2$. In fact, in this case, we have  
\begin{align*}
\int_{\Omega}\arrowvert \nabla u_i\arrowvert^2\le \bigg[\int_{\Omega} (-\Delta u_i)^2\bigg]^{\frac{1}{2}}&\bigg[\int_{\Omega}u_i^2\bigg]^{\frac{1}{2}}\\
 {} & =\bigg[\lambda_i-\int_{\Omega}qu_i^2\bigg]^{\frac{1}{2}}\leq \Big(\overline{\lambda_i}\Big)^{\frac{1}{2}},
 \end{align*}
where $\overline{\lambda_i}=\lambda_i-\inf_{\Omega} q.$\\
Taking into account this modification in inequalities (\ref{m3}) and (\ref{o3}), we obtain inequality (\ref{16'''}).
\item If $f(x)=g(x)=(\lambda_{k+1}-x)^2$, then inequality (\ref{14'}) extends inequality (\ref{WX2}) of Wang and Xia to any Riemannian submanifolds of $\R^m$. We also observe that, by using a Chebyshev inequality (for instance the one of Lemma 1 in \cite{Cheng-Ichik-Mamet1}), inequality (\ref{CIM}) of Cheng, Ishikawa and Mametsuka  can be easily deduced from inequality (\ref{14'}). 
\item If $f(x)=g(x)^2=(\lambda_{k+1}-x)$, then inequality (\ref{14'}) generalizes inequality (\ref{1}) of Cheng and Yang to the case of Euclidean submanifolds.
\end{itemize}
\end{remark}
Using the standard emdeddings of the rank one compact symmetric spaces in a Euclidean space (see for instance Lemma 3.1 in \cite{Soufi.Harl.Ilias} for the values of $|H|^2$ of these embeddings), we can extend easily the previous theorem to domains or submanifolds of these symmetric spaces and obtain 
\begin{theorem}\label{clamp2}
Let $\bar{M}$ be the sphere $\mathbb{S}^m$, the real projective space $\R P^m$, the complex projective space $\C P^m$ or the quaternionic projective space $\mathbb{Q} P^m$ endowed with their respective metrics. Let $(M,g)$ be a compact Riemannian manifold of dimension $n$ and let $X : M\longrightarrow \bar{M}$ be an isometric immersion of mean curvature $H$. Consider the clamped plate problem on a bounded domain $\Omega$ of $M$. For any $k\ge 1$ such that $\lambda_{k+1}>\lambda_k$ and for any $(f,g) \in \Im_{\lambda_{k+1}}$, we have
\begin{align}\label{r3}
\sum_{i=1}^kf(\lambda_i) \leq \frac{2}{n}\bigg\{\sum_{i=1}^k g(\lambda_i)\Big[&2(n+2)\lambda_i^{\frac{1}{2}}+n^2\delta'\Big]\bigg\}^{\frac{1}{2}}\times\nonumber\\
&\bigg\{\sum_{i=1}^k\frac{\big(f(\lambda_i)\big)^2}{g(\lambda_i)(\lambda_{k+1}-\lambda_i)}
\Big[\lambda_i^{\frac{1}{2}}+\frac{n^2}{4}\delta'\Big]\bigg\}^{\frac{1}{2}},
\end{align}
where \begin{equation*}
    \delta'=sup(\arrowvert H\arrowvert^2+1)\,\,\,     \text{if $\overline{M}=\mathbb{S}^{m}$}, 
    \end{equation*}
\begin{equation*}\delta'=sup(\arrowvert H\arrowvert^2+d(n)),\,\, \hbox{where}\,\, d(n)=
 \begin{cases}
 \frac{2(n+1)}{n}   &\text{if $\overline{M}= \mathbb{R}P^{m}$},\\
  \frac{2(n+2)}{n}   &\text{if $\overline{M}= \mathbb{C}P^{m}$},\\
    \frac{2(n+4)}{n}    &\text{if $\overline{M}= \mathbb{Q}P^{m}$}.\\
    \end{cases}\end{equation*}  
    \end{theorem}
\begin{remark}
\begin{itemize}
\item We observe (as in Remark 3.2 of \cite{Soufi.Harl.Ilias}) that in some special geometrical situations, the constant $d(n)$ in the inequality of Theorem~\ref{clamp2} can be replaced by a sharper one. For instance, when $\bar M=\C P^m$ and
\begin{itemize}
       \item [-] $M$ is odd--dimensional, then one can replace  $d(n)$ by $d'(n)= \frac{2}{n}(n+2-\frac 1 n)$,
       \item [-] $X(M)$ is totally real, then  $d(n)$ can be replaced by $d'(n)= \frac{2(n+1)}{n}$.
       \end{itemize}
\item When $f(x)=g(x)=(\lambda_{k+1}-x)^2$, and $\bar{M}$ is a sphere, (\ref{r3}) generalizes to submanifolds inequality (\ref{WX}) established by Wang and Xia for spherical domains.
\item As for Theorem \ref{clamp1}, the result of Theorem \ref{clamp2} holds for a more general operator $\Delta^2+q$, with the same modification (i.e $\bar{\lambda_{i}}^{\frac{1}{2}}$ instead of $\lambda_{i}^{\frac{1}{2}}$). 
\end{itemize}
\end{remark}
%%%%%%%%%%%%%%%%%%%%%%%%%%%%%%%%%%%%%%%%%%%%%%%%%%%%%%%%%%%%%%%%%%%%%%%%%%
\section{Manifolds admitting spherical eigenmaps}
In this section, as before, we let $(M,g)$ be a Riemannian manifold and $\Omega$ be a bounded domain of $M$. A map $X:(M,g) \rightarrow \mathbb{S}^{m-1}$ is called an eigenmap if its components $X_1,\,X_2,\ldots,X_{m}$ are all eigenfunctions associated to the same eigenvalue $\lambda$ of the Laplacian of $(M,g)$. This is equivalent to say that the map $X$ is a harmonic map from $(M,g)$ into $\mathbb{S}^{m-1}$ with constant energy $\lambda$  $\Big({\rm i.e}\;\sum_{p=1}^{m}\left|\nabla X_p \right|^2=\lambda$\Big). The most important examples of such manifolds $M$ are the compact homogeneous Riemannian manifolds. In fact, they admit eigenmaps for all the positive eigenvalues of their Laplacian (see \cite{Li}).

\begin{theorem}\label{eigenmap}
Let $\lambda$ be an eigenvalue of the Laplacian of $(M,g)$ and suppose that $(M,g)$ admits an eigenmap $X$ associated to this eigenvalue $\lambda$. Let $\Omega$ be a bounded domain of $M$ and consider the clamped plate problem (\ref{clamped prob}) on it. For any $k\ge 1$ such that $\lambda_{k+1}>\lambda_k$ and for any $(f,g) \in \Im_{\lambda_{k+1}}$, we have 
\begin{align}\label{h}
\sum_{i=1}^kf(\lambda_i)\nonumber\\ \leq \bigg[\sum_{i=1}^kg(\lambda_i)\Big(\lambda+6\lambda_i^{\frac{1}{2}}\Big)\bigg]^{\frac{1}{2}}&\bigg[\sum_{i=1}^k\frac{\big(f(\lambda_i)\big)^2}{g(\lambda_i)(\lambda_{k+1}-\lambda_i)}\Big(\lambda+4\lambda_i^{\frac{1}{2}}\Big)\bigg]^{\frac{1}{2}}.
\end{align}
\end{theorem}
\begin{proof}
As in the proof of Theorem \ref{clamp1}, we apply Theorem \ref{abstract form 1} with $A=\Delta^2$, $B_p=X_p$ and $T_p=[\Delta,X_p]$, $p=1,\ldots,m$, to obtain
\begin{align}\label{9'}
\displaystyle\bigg( \sum_{i=1}^k
\sum_{p=1}^mf(\lambda_i&)\langle[[\Delta,X_p],
X_p]u_i,u_i\rangle_{L^2}\bigg)^2 \nonumber\\
\leq 4\bigg[\displaystyle \sum_{i=1}^k
\sum_{p=1}^m g(\lambda_i) \langle [\Delta^2,X_p]u_i, X_p
u_i\rangle_{L^2}\bigg] & \bigg[\displaystyle \sum_{i=1}^k \sum_{p=1}^m
\frac{\big(f(\lambda_i)\big)^2}{g(\lambda_i)(\lambda_{k+1}-\lambda_i)} \|[\Delta,X_p] u_i\|_{L^2}^2\bigg],
\end{align}
where $\{u_i\}_{i=1}^{\infty}$ is a complete $L^2$-orthonormal basis of eigenfunctions of $\Delta^2$ associated to $\{\lambda_i\}_{i=1}^{\infty}$.
 As before, we have
\begin{equation*}
\sum_{p=1}^m \langle [[\Delta,X_p],X_p]u_i,u_i\rangle_{L^2}
= 2\sum_{p=1}^m \int_{\Omega} \arrowvert \nabla X_p \arrowvert^2u_i^2.
\end{equation*}
Since  \begin{equation}\label{6'}\sum_{p=1}^m\arrowvert \nabla X_p\arrowvert^2=\lambda,\, \sum_{p=1}^m X_p^2=1\,\, \hbox{and}\,\, \Delta X_p=-\lambda X_p, \end{equation} we have
\begin{equation}\label{i}
\sum_{p=1}^m \langle [[\Delta,X_p],X_p]u_i,u_i\rangle_{L^2}
=2\lambda
\end{equation}
and
\begin{align}
\sum_{p=1}^m\|[\Delta,X_p] u_i\|_{L^2}^2&=\sum_{p=1}^m\int_{\Omega}\Big([\Delta,X_p]u_i\Big)^2 \label{j}\\
=4\int_{\Omega}\sum_{p=1}^m(\nabla X_p.\nabla u_i)^2+\int_{\Omega}\sum_{p=1}^m& (\Delta X_p)^2 u_i^2+4\int_{\Omega}\sum_{p=1}^m u_i \Delta X_p \nabla X_p.\nabla u_i\nonumber\\
\leq \label{z} 4\int_{\Omega}\sum_{p=1}^m \arrowvert \nabla X_p\arrowvert^2\arrowvert\nabla u_i\arrowvert^2+\lambda^2\int_{\Omega}&\Big(\sum_{p=1}^m X_p^2\Big)u_i^2-2\lambda\int_{\Omega} u_i\nabla\Big(\sum_{p=1}^m X_p^2\Big).\nabla u_i\\
= 4\lambda \int_{\Omega}& (-\Delta u_i)u_i +\lambda^2\nonumber\\
\leq 4\lambda \Big(\int_{\Omega}& (-\Delta u_i)^2\Big)^{\frac{1}{2}} \Big(\int_{\Omega}u_i^2\Big)^{\frac{1}{2}} +\lambda^2\label{5'}\\
= 4\lambda \lambda_i^{\frac{1}{2}}&+\lambda^2 \label{k},
    \end{align}
where we used the Cauchy-Schwarz inequality to obtain (\ref{z}) and (\ref{5'}).\\ Similarly we infer, from identities (\ref{l3}) and (\ref{6'}), 
\begin{align}
 \sum_{p=1}^m \langle[\Delta^2,X_p]u_i,X_p u_i\rangle_{L^2}
=&\lambda^2 \int_{\Omega}u_i^2 -\lambda\int_{\Omega}\nabla\Big(\sum_{p=1}^m X_p^2\Big).\nabla u_i^2\nonumber\\ &+4\sum_{p=1}^m\int_{\Omega}\Big(\nabla X_p.\nabla u_i\Big)^2+2\lambda\int_{\Omega}(-\Delta u_i)u_i\nonumber\\
\leq \lambda^2 +4\int_{\Omega}\sum_{p=1}^m&\arrowvert\nabla X_p\arrowvert^2\arrowvert\nabla u_i\arrowvert^2+2\lambda\Big(\int_{\Omega}\big(-\Delta u\big)^2\Big)^{\frac{1}{2}}\Big(\int_{\Omega}u_i^2\Big)^{\frac{1}{2}}\nonumber\\
\leq&\lambda^2+4\lambda\lambda_i^{\frac{1}{2}}+2\lambda\lambda_i^{\frac{1}{2}}\nonumber\\
=&\lambda^2+6\lambda\lambda_i^{\frac{1}{2}} \label{m},
\end{align}
Incorporating (\ref{i}), (\ref{k}) and (\ref{m}) in inequality (\ref{9'}), we get the statement of the theorem.
\end{proof}
An immediate consequence of Theorem \ref{eigenmap} is the following
\begin{corollary}
Let $(M,g)$ be a compact homogeneous Riemannian manifold without boundary and let $\lambda_{1}$ be the first non zero eigenvalue of its Laplacian . Then inequality of Theorem \ref{eigenmap} holds with $\lambda=\lambda_{1}$.
\end{corollary}
\begin{remark}
As before, one can get a similar result for the operator $\Delta^2+q$. 
\end{remark}
%%%%%%%%%%%%%%%%%%%%%%%%%%%%%%%%%%%%%%%%%%%%%%%%%%%%%%%%%%%%%%%%%%%%%%%%%%%%%%%%%%%
\section{domains in the hyperbolic space}
We turn next to the case of a domain $\Omega$ of a hyperbolic space. It is easy to establish a universal inequality for eigenvalues of the clamped plate problem (\ref{clamped prob})on $\Omega$ in the vein of the preceding ones. Unfortunately, until now we have not succeeded in obtaining a simple generalization for the case of domains of hyperbolic submanifolds. In what follows, we take the half-space model for $\mathbb{H}^n$ i.e
\begin{equation*}
 \mathbb{H}^n=\{x=(x_1,x_2,\ldots,x_n) \in \R^n;\,x_n > 0 \}
\end{equation*}
 with the standard metric
\begin{equation*}
 ds^2=\frac{dx_1^2+dx_2^2+\ldots+dx_n^2}{x_n^2}.
\end{equation*}
We note that in terms of the coordinates $(x_{i})_{i=1}^{n}$, the Laplacian of $\mathbb{H}^n$ is given by
\begin{equation*}
 \Delta=x_n^2\sum_{j=1}^n \frac{\partial^2}{\partial x_j \partial x_j}+(2-n)x_n\frac{\partial}{\partial x_n}.
\end{equation*}
\begin{theorem}\label{HYP} For any $k\ge 1$ such that $\lambda_{k+1}>\lambda_k$, the eigenvalues $\lambda_i's$ of the clamped problem (\ref{clamped prob}) on the bounded domain $\Omega$ of $\mathbb{H}^n$ must satisfy, for any $(f,g)  \in \Im_{\lambda_{k+1}}$,
\begin{align}
 \sum_{i=1}^k f(\lambda_i) \leq\bigg[\sum_{i=1}^k&g(\lambda_i)\Big(6\lambda_i^{\frac{1}{2}}-(n-1)^2\Big)\bigg]^{\frac{1}{2}}\times\nonumber\\ &\bigg[\sum_{i=1}^k\bigg(\frac{(f(\lambda_i))^2}{g(\lambda_i)(\lambda_{k+1}-\lambda_i)}\bigg)\Big(4\lambda_i^{\frac{1}{2}}-(n-1)^2\Big)\bigg]^{\frac{1}{2}},
\end{align}
\end{theorem}
\begin{proof}
Theorem \ref{abstract form 1} remains valid for $A=\Delta^2$, $B_p=F= \ln x_n$ and $T_p=[\Delta,F]$, for all $p=1,\ldots,n$. Thus, denoting by $u_i$ the eigenfunction corresponding to $\lambda_i$, we have
\begin{align}\label{12'}
\displaystyle\bigg(\sum_{i=1}^k
f(\lambda_i)&\langle[[\Delta,F],
F]u_i,u_i\rangle_{L^2}\bigg)^2\nonumber\\ \leq \displaystyle4\bigg[\sum_{i=1}^kg(\lambda_i) \langle [\Delta^2,F]u_i, F
u_i\rangle_{L^2}\bigg]&\bigg[\displaystyle \sum_{i=1}^k
 \bigg(\frac{(f(\lambda_i))^2}{g(\lambda_i)(\lambda_{k+1}-\lambda_i)}\bigg)\|[\Delta,F] u_i\|_{L^2}^2\bigg].
\end{align}
Let us start by the calculation of
\begin{align*}
\langle [[\Delta,F],F]u_i,u_i \rangle_{L^2}=& \int_{\Omega}\bigg([\Delta,F](Fu_i)-F[\Delta,F]u_i\bigg)u_i\nonumber\\
=&\int_{\Omega}\Big(\Delta(F^2u_i)-2F\Delta(Fu_i)+F^2\Delta u_i\Big)u_i.
\end{align*}
We note that \begin{equation}\label{10'}\Delta F =1-n\,\,\hbox{and} \,\,|\nabla F|^2=1.\end{equation} Thus a direct calculation gives
\begin{align}\label{e}
\langle [[\Delta,F],F]u_i,u_i \rangle_{L^2}= 2\int_{\Omega}|\nabla F|^2u_i^2=2.
\end{align}
On the other hand, using again identities (\ref{10'}), we obtain
\begin{align}
\|[\Delta,F]u_i\|_{L^2}^2=&\int_{\Omega}(\Delta F u_i+2 \nabla F.\nabla u_i)^2\nonumber\\
=&\int_{\Omega}(\Delta F)^2u_i^2+4\int_{\Omega}(\nabla F.\nabla u_i)^2+4\int_{\Omega}\Delta F u_i \nabla F.\nabla u_i\\
=&(1-n)^2+4\int_{\Omega}(\nabla F.\nabla u_i)^2+4(1-n)\int_{\Omega}u_i \nabla F.\nabla u_i.\label{11'}
\end{align}
But \begin{equation*}\int_{\Omega} u_i\nabla F.\nabla u_i=-\int_{\Omega} u_i\nabla F.\nabla u_i - \int_{\Omega}u_i^2 \Delta F,\end{equation*}
hence \begin{equation}\label{c}\int_{\Omega}u_i \nabla F.\nabla u_i=\frac{n-1}{2}.\end{equation} Then we infer, from (\ref{10'}), (\ref{11'}) and (\ref{c}),
\begin{align}
\|[\Delta,F]u_i\|_{L^2}^2 \leq & -(n-1)^2+4\int_{\Omega}|\nabla F|^2|\nabla u_i|^2 \nonumber \\
=&-(n-1)^2+4\int_{\Omega}|\nabla u_i|^2 \nonumber\\
=&-(n-1)^2+4\int_{\Omega}u_i(-\Delta u_i)\nonumber\\
\leq & -(n-1)^2+4 \bigg(\int_{\Omega}u_i^2\bigg)^{\frac{1}{2}}\bigg(\int_{\Omega}(-\Delta u_i)^2\bigg)^{\frac{1}{2}}\label{a}\\
=& 4\lambda_i^{\frac{1}{2}}-(n-1)^2\label{b}
\end{align}
Now,
\begin{align}
[\Delta^2,F]u_i=&\Delta^2(Fu_i)-F\Delta^2u_i\nonumber\\
=&\Delta(\Delta F u_i+2\nabla F.\nabla u_i+F\Delta u_i)-F\Delta^2u_i \nonumber\\
=&2(1-n)\Delta u_i+2\Delta(\nabla F.\nabla u_i)+2 \nabla F.\nabla \Delta u_i,
\end{align}
 thus
\begin{align*}
 \langle[\Delta^2,F]u_i&,Fu_i\rangle_{L^2}\nonumber\\
=2(1-n)\int_{\Omega}Fu_i\Delta u_i+2\int_{\Omega}Fu_i\Delta&(\nabla F.\nabla u_i)+2\int_{\Omega}Fu_i\nabla F.\nabla \Delta u_i\nonumber\\
=2(1-n)\int_{\Omega}Fu_i\Delta u_i+2\int_{\Omega}\Delta (Fu_i)&\nabla F.\nabla u_i -2\int_{\Omega}div(Fu_i\nabla F)\Delta u_i\nonumber\\
=2\int_{\Omega}\Delta Fu_i\nabla F.\nabla u_i +4\int_{\Omega}(&\nabla F.\nabla u_i)^2-2\int_{\Omega}|\nabla F|^2u_i \Delta u_i.
\end{align*}
We infer, from (\ref{10'}) and (\ref{c}),
\begin{align}\label{d}
\langle[\Delta^2,F]u_i,Fu_i\rangle_{L^2}\leq & -(n-1)^2+4\int_{\Omega}|\nabla F|^2|\nabla u_i|^2+2\int_{\Omega}u_i(-\Delta u_i) \nonumber\\
=&-(n-1)^2+6\int_{\Omega}u_i(-\Delta u_i)\nonumber\\
\leq & 6\bigg(\int_{\Omega}u_i^2\bigg)^{\frac{1}{2}}\bigg(\int_{\Omega}(-\Delta u_i)^2\bigg)^{\frac{1}{2}}-(n-1)^2\nonumber\\
=&6 \lambda_i^{\frac{1}{2}}-(n-1)^2.
\end{align}
Inequality (\ref{12'}) along with (\ref{e}), (\ref{b}) and (\ref{d}) gives the statement of the theorem.
\end{proof}
\begin{remark} \begin{itemize}
\item It will be interesting to look for an extension of Theorem \ref{HYP} to domains of hyperbolic submanifolds.
\item Note that our method works for any bounded domain $\Omega$ of a Riemannian manifold admitting a function such that $|\nabla h|$ is constant and $|\Delta h|\le C$, where $C$ is a constant.

\item As before, we observe that we have the same statement as in Theorem \ref{HYP} for the operator $\Delta^{2}+q$ (it suffices to replace $\lambda_{i}^{\frac{1}{2}}$ by $\overline{\lambda_{i}}^{\frac{1}{2}}$).
\end{itemize}\end{remark}

\end{document}